\documentclass[11pt]{birkjour}
\usepackage{hyperref}
\usepackage{url}
\hypersetup{
	colorlinks=true,
	linkcolor=blue,
	citecolor=blue,
	filecolor=blue,
	urlcolor=blue,
	pdftitle={Overleaf Example},
	pdfpagemode=FullScreen,
}
\newtheorem{thm}{Theorem}[section]

\newtheorem{lem}[thm]{Lemma}
\newtheorem{prop}[thm]{Proposition}
\theoremstyle{definition}

\newtheorem{nota}[thm]{Notation}
\theoremstyle{remark}
\newtheorem{rem}[thm]{Remark}

\numberwithin{equation}{section}
\begin{document}
 \author[Ajebbar Omar, Elqorachi Elhoucien and Jafar Ahmed]{Ajebbar Omar$^{1}$, Elqorachi Elhoucien$^{2}$ and Jafar Ahmed$^{2}$}

 \address{%
 	$^1$
 	Sultan Moulay Slimane University, Multidisciplinary Faculty,
 	Department of mathematics,
 	Beni Mellal,
 	Morocco}
\address{%
 	$^2$
 	Ibn Zohr University, Faculty of sciences,
 	Department of mathematics,
 	Agadir,
 	Morocco}
 \email{omar-ajb@hotmail.com, elqorachi@hotmail.com, hamadabenali2@gmail.com}
	 	
\title[Integral Kannappan-Sine subtraction and addition law on semigroups]{Integral Kannappan-Sine subtraction and addition law on semigroups}
	
	 	\begin{abstract}
		  Let  $S$ be a semigroup, $\mu$   a
		 discrete measure   on $S$ and  $\sigma:S \longrightarrow S$  is an involutive automorphism.
		We determine the  complex-valued solutions of the integral Kannappan-Sine subtraction law  $$\int_{S}f(x\sigma(y)t)d\mu(t)=f(x)g(y)-f(y)g(x),\; x,y \in S,$$ and the integral Kannappan-Sine addition law $$\int_{S}f(x\sigma(y)t)d\mu(t)=f(x)g(y)+f(y)g(x),\; x,y \in S.$$
		 We express the solutions   by means of exponentials on S,  the solutions of the special sine addition law $f(xy)=f(x)\chi(y)+f(y)\chi(x),$ $x,y\in S$ and the solutions of  of the special case of the integral Kannappan-Sine addition law  $\int_{S}f(x\sigma(y)t)d\mu(t)=[f(x)\chi(y)+f(y)\chi(x)]\int_{S}\chi(t)d\mu(t), $ $x,y\in S$, and where    $\chi$: $S\longrightarrow \mathbb{C}$ is an exponential. The continuous solutions on topological semigroups are also given.
	 \end{abstract}
	\maketitle
	\section{Introduction}
Let  $S$ be a semigrpoup, $\mu$ a discrete measure on $S$ and  $\sigma:S\longrightarrow S$ is an  involutive automorphism. That $\sigma$ is involutive means that $\sigma\circ\sigma(x)=x$ for all $x\in S.$\\
The functional equations
\begin{equation}
\label{7} f(x\sigma(y))=f(x)g(y)-f(y)g(x),\;\; x,y \in {S},\end{equation} respectively,
\begin{equation}
\label{elqorachi}f(x\sigma(y))=f(x)g(y)+f(y)g(x),\;\; x,y \in {S},\end{equation} for functions $f, g$ : $S \longrightarrow \mathbb{C}$ are called the subtraction law for the sine, resp. the addition law for the sine   because of
its solution $f = \sin$, $g = \cos$ in the case $S = (\mathbb{R}, +)$ and $\sigma=-Id$. The sine addition and subtraction formula have attracted much interest. Equation (\ref{7}) and (\ref{elqorachi}) has been solved by Poulsen and Stetk\ae r \cite{ebb} on groups, monoids generated by their  squares \cite[Proposition 3.6]{eb}, semigroups generated by their squares \cite[Theorem 5.2]{ajj} and  general monoids \cite[Theorem 4.2]{c}. The most current results about (\ref{7}) and  (\ref{elqorachi})  on general semigroups are \cite[Theorem 4.2]{vb},
\cite[Theorem 4.3  and Theorem 5.2]{as} and \cite[Theorem 3.1]{d}. For other pertinent literature  about these functional equations we refer to  \cite[Ch.13]{a}, \cite[Chapter 4]{g} and their references. Thus equations  (\ref{7}) and    (\ref{elqorachi}) are solved in large generality,  and it makes sense to introduce and solve other like  trigonometric functional equations on semigroups, and expressing their solutions in terms of solutions of (\ref{7}) and (\ref{elqorachi}).\\\\
In 2006, Elqorachi and Redouani \cite{Redouani}  studied the bounded
and continuous solutions of the following trigonometric functional equations \begin{equation}\label{000}
\int_{S}f(xt\sigma(y))d\mu(t)=f(x)g(y)\pm f(y)g(x), \;x,y \in S,\end{equation}\begin{equation}\label{00}\int_{S}f(xt\sigma(y))d\mu(t)=f(x)f(y)+g(y)g(x), \;x,y \in S
\end{equation} on locally compact groups, and where $\mu$ is a bounded complex $\sigma$-invariant measure on $S$. The solutions are expressed by means of $\mu$-spherical functions: $\int_{S}\phi(xty)d\mu(t)=\phi(x)\phi(y),\;x,y\in S$ and solutions of the functional equation $\int_{S}f(xty)d\mu(t)=f(x)\phi(y)+ f(y)\phi(x), \;x,y \in S$. We refer also to \cite{akk1}, \cite{akk2} and \cite{akk3}.
\\\\In  \cite[Proposition 16]{i} Stetk\ae r solved the Kannappan multiplicative functional equation
\begin{equation}\label{4}
f(xyz_{0})=f(x)f(y),\,x,y \in S,
\end{equation}
on semigroups, and where $z_{0}$ is a fixed element in $ S$. The solutions of (\ref{4}) are of the form $f=\chi(z_0)\chi,$ where $\chi:$ $S\to \mathbb{C}$ is a multiplicative function.\\In
\cite[Corollary 2.5]{red} a similar equation \begin{equation}\label{elqorachi1}
\int_{S}f(xyt)d\mu(t)=f(x)f(y),\,x,y \in S,
\end{equation} was solved on semigroups.  The solutions formulas of (\ref{elqorachi1}) are of the form: $f=[\int_{S}\chi(t)d\mu(t)]\chi,$ where $\chi:$ $S\to \mathbb{C}$ is a multiplicative function.
\\In the present paper we shall consider the generalization
\begin{equation}
\label{6}
\int_{S}f(x\sigma(y)t)d\mu(t)=f(x)g(y)-f(y)g(x), \;x,y \in S,
\end{equation}  \begin{equation}
\label{elqorachi2}
\int_{S}f(x\sigma(y)t)d\mu(t)=f(x)g(y)+f(y)g(x), \;x,y \in S.
\end{equation}  If $S$ is a monoid with identity element $e$, then by replacing $y$ by the identity element $e$ in (\ref{6}) and in (\ref{elqorachi2}),   equations (\ref{6}) and (\ref{elqorachi2}) can be written as follows:

$$ g(e) f(x\sigma(y))=f(x)g(y)\mp f(y)g(x),\,x,y\in S$$ with $f(e)=0$ or \begin{equation}\label{key1}
f(x\sigma(y))=f(x)g(y)\mp f(y)g(x)\pm\dfrac{f(e)}{g(e)} g(x\sigma(y)),\,x,y\in S
\end{equation} with $f(e)\neq 0$ and $g(e)\neq 0.$
Furthermore, the second equation (\ref{key1}) with $(-)$: $f(x\sigma(y))=f(x)g(y)- f(y)g(x)+\dfrac{f(e)}{g(e)} g(x\sigma(y)),\,x,y\in S$ reduce to the sine subtraction law
\begin{equation}\label{ast}
F(x\sigma(y))=F(x)g(y)-F(y)g(x), \;x,y \in S,
\end{equation} where $F=\dfrac{g(e)}{f(e)}f -g$,  and then from \cite[Theorem 4.2 ]{c} we derive explicit formulas
for $f$ and $g$ of equation (\ref{6}) on monoids. The solutions of  (\ref{6}) were given by Zeglami et al. assuming that $S$ is a
group     \cite[Theorem 4.2]{l}. On the other hand, equation (\ref{key1}) with $(+)$: $f(x\sigma(y))=f(x)g(y)+ f(y)g(x)-\dfrac{f(e)}{g(e)} g(x\sigma(y)),\,x,y\in S$ is  exactly  the cosine subtraction law
\begin{equation}\label{ast1}
F(x\sigma(y))=F(x)F(y)-(F-g)(y)(F-g)(x), \;x,y \in S,
\end{equation} where $F=\frac{1}{2}(g+\frac{g(e)}{f(e)}f)$, and  which was solved  on semigroups (see for example \cite[Theorem 3.3]{as}).\\\\The purpose of the present paper is to show  how the above relation between (\ref{6}), (\ref{elqorachi2}) and equation (\ref{7}), (\ref{elqorachi}) and (\ref{ast})  on monoids extends to the much wider framework of (\ref{6}) and (\ref{elqorachi2}) on semigroups.  We show that any solution of (\ref{6}), (\ref{elqorachi2})  is closely related to the solutions of (\ref{7})  and (\ref{elqorachi}), and so can be expressed in terms of exponentials and solutions of same special trigonometric functional equations on  semigroup  $S.$\\\\Kannappan's functional equations \cite{PL} and their extensions are treated  in several works. We refer for example to \cite{boui},  \cite{red}, \cite{f}, \cite{ff}, \cite{pe}, \cite{i} and \cite{l}.\\\\
This paper extends our previous results about the functional equations:\\ $f (x\sigma(y)z_{0})= \int_{S}f(x\sigma(y)t)d\delta_{z_0}(t)=f(x)g(y)\mp f(y)g(x)$  on  semigroups (\cite{f}, \cite{ibti}, \cite{ff}).
 Here we find the general solution for all semigroups and for all discrete measure  $\mu=\sum_{i\in I} \alpha_{i}\delta_{z_i}$.
\\\\Our main contributions to the knowledge about  integral Kannappan-Sine subtraction law (\ref{6}) and integral Kannappan-Sine addition  law   (\ref{elqorachi2}) are the following\\
- We extend the setting from groups to semigroups with involution.\\- We relate the solutions of equations (\ref{6}), (\ref{elqorachi2}) to those of equation (\ref{7}), (\ref{elqorachi}) and (\ref{44}).\\-We derive formulas for solutions of (\ref{6}) (Section 3, Theorem 3.6).\\
-We express the solutions of (\ref{elqorachi2}) ( Section 4,  Theorem \ref{t1} ) in term of  exponentials  and the solutions of the following special case of   the  integral Kannappan-Sine  addition law, namely
\begin{equation}
\label{44}
\int_{S} f(x\sigma(y)t)d\mu(t)= f(x)\chi(y)\int_{S} \chi(t)d\mu(t)+f(y)\chi(x)\int_{S} \chi(t)d\mu(t),
\end{equation}
   for all $x, y \in S$, where $\chi$ is an exponential  such that $\int_{S} \chi(t)d\mu(t)\neq0$.\\ It turns out that, like on abelian groups, only multiplicative functions and solutions of some special trigonometric functional equations occur in the solution formulas for $f$ and $g$, and no non-abelian phenomena like group representations crop up but contrasts results for the functional equations (\ref{000}), (\ref{00}). In particular, the solutions of our functional equations are abelian, except for some arbitrary functions.\section{Notation and terminology}
    A     semigroup $S$ is a set with an associative composition rule. A monoid is a semigroup with an identity element. If $S$ is a toplogical semigroup then $C(S)$ denotes the algebra of continuous functions mapping $S$ into $\mathbb{C}$. Let $\mathbb{C}^{*}=\mathbb{C}\setminus\{0\}$. \\ A function $\chi$ : $\longrightarrow \mathbb{C}$ is multiplicative if $ \chi(xy) = \chi(x)\chi(y)$ for all $x, y \in
S$. If in addition $\chi \neq 0$ then we call $\chi$ an \textit{exponential}. \\
For convenience we introduce the following notations:
\begin{nota}
	 Let $\phi_\chi :  S\longrightarrow \mathbb{C}$   denotes  a solution  of  the  following special sine addition law
\begin{equation}\label{111}
\phi_\chi(xy) = \phi_\chi (x)\chi(y) + \chi(x)\phi_\chi (y),\,x, y \in S,
\end{equation}
where  $\chi$ is  an exponential on $S$. The formulas for the function   $\phi_\chi$  are known and can be found in
\cite[Theorem 3.1 (B)]{d}.
\end{nota}
\begin{nota}
Let $\Phi_\chi$ : $S\longrightarrow \mathbb{C}$ denotes a solution of special integral Kannappan-Sine addition law  (\ref{44}).
\end{nota}

\section{Solutions of equation (1.7)}
		In the following proposition we give the solutions formulas for  the functional equation
		\begin{equation}
		\label{04}
		\int_{S}f(x\sigma(y)t)d\mu(t)=f(x)f(y), \; x,y \in S
		\end{equation}
		which  has been solved on semigroups   by Elqorachi \cite[Corollary 2.5 ]{red} for  $\sigma=Id$.
	 	\begin{prop}
	 		\label{p1}
	 	Let  $f  :S\longmapsto\mathbb{C}$  be a solution      the functional equation (\ref{04}). Then we have the following.  \\
	 	(a) 	$\int_{S}f(t)d\mu(t)\neq0 $ if and only if  $f\ne0.$\\
	  (b)	$f=[\int_{S}\chi(t)d\mu(t)]\chi,$	
	 	  where    $\chi:S\longrightarrow \mathbb{C}$ is a  multiplicative function such that  	$\chi\circ \sigma =\chi$.
	 	\end{prop}
	 	\begin{proof}
	 		(a)	If $f=0$ then   $\int_{S}f(t)d\mu(t)=0$. Conversely, if $\int_{S}f(t)d\mu(t)=0$ then,
by replacing $x$ by $xs$ and $y$ by $xk$ in (\ref{04}) and integrating the two members of equation with respect to
	 	 $s$ and $k$ we obtain
	 		\begin{equation}
	 		\label{A1}
	 		\int_{S}\int_{S}\int_{S}f(xs\sigma(xk)t)d\mu(s)d\mu(k)d\mu(t)= \left( \int_{S}f(xk)d\mu(k) \right) ^{2}.
	 		\end{equation}
By using (\ref{04}) the left hand side of (\ref{A1}) can be written as
\begin{equation*}\begin{split}& 		\int_{S}\int_{S}\int_{S}f(xs\sigma(xk)t)d\mu(s)d\mu(k)d\mu(t)\\
&=\int_{S}\int_{S}\int_{S}f(xs\sigma(x)\sigma(k)t)d\mu(s)d\mu(k)d\mu(t)\\
&=\int_{S}\int_{S}\left[ \int_{S}f(xs\sigma(x)\sigma(k)t)d\mu(t)\right] d\mu(s)d\mu(k)\\	 		&=\int_{S}f(xs\sigma(x))d\mu(s)\int_{S}f(k)d\mu(k)\\
&=0.
\end{split}\end{equation*}
	 		In view of  (\ref{A1}) we deduce that
	 		$  \int_{S}f(xk)d\mu(k)=0$ for all $x\in S$, and  from equation (\ref{04}) we get
	 		$\int_{S}f(x\sigma(y)t)d\mu(t)=f(x)f(y)=0$ for all $x,y\in S.$ This implies that $f=0$. \\
(b) Assume that $f\neq0$.  By replacing $y$ by $yzs$ in (\ref{04}) and integrating the result obtained with respect to $s$  we get that
\begin{equation*}\begin{split}&	 		\int_{S}\int_{S}f(x\sigma(yzs)t)d\mu(s)d\mu(t)=\int_{S}\left[  \int_{S}f(x\sigma(y)\sigma(zs)t)d\mu(t)\right] d\mu(s)\\
&\quad\quad\quad\quad\quad\quad\quad\quad\quad\quad\quad\quad\,\,\,\,
=f(x\sigma(y))\int_{S}f(zs)d\mu(s)
\end{split}\end{equation*} and
\begin{equation*}\begin{split}&	 	\int_{S}\int_{S}f(x\sigma(yzs)t)d\mu(s)d\mu(t)=\int_{S}\left[    \int_{S}f(x\sigma(yz)\sigma(s)t)d\mu(t)\right] d\mu(s)\\
&\quad\quad\quad\quad\quad\quad\quad\quad\quad\quad\quad\quad
\,\,\,\,=f(x\sigma(yz))\int_{S}f(s)d\mu(s).
\end{split}\end{equation*}
So,
\begin{equation}\label{key}
f(x\sigma(y))\int_{S}f(zs)d\mu(s)=	f(x\sigma(yz))\int_{S}f(s)d\mu(s)
\end{equation}for all $x,y,z\in S.$
Now, replacing $z$ by $\sigma(k)$  in (\ref{key}) and integrating the result obtained with respect to $k$
we obtain
\begin{equation}\label{09}\begin{split}&	
f(x \sigma(y)) \int_{S}\int_{S}f(\sigma(k)s)d\mu(k)d\mu(s)=	\int_{S}f(x\sigma(y)k)d\mu(k)\int_{S}f(s)d\mu(s)\\
&\quad\quad\quad\quad\quad\quad\quad\quad\quad\quad\quad\quad
\quad\quad\quad\,=f(x)f(y)\int_{S}f(s)d\mu(s).
\end{split}\end{equation}
Since  $\int_{S}f(s)d\mu(s)\neq 0$ and  $f\neq 0$, we get that  $\int_{S}\int_{S}f(\sigma(k)s)d\mu(k)d\mu(s)\neq0$. Hence equation (\ref{09}) can be written as follows
	 		\begin{equation}
	 		\label{333}
	 		f(x\sigma(y) )=\alpha f(x)f(y),
	 		\end{equation}
	 		where $\alpha:= \dfrac{\int_{S}f(s)d\mu(s)}{\int_{S}\int_{S}f(\sigma(k)s)d\mu(k)d\mu(s)}\neq0$. From equation  (\ref{333}) we get 	
$f(x\sigma(yz) )= \alpha f (x)f(yz)=f((x\sigma(y)\sigma(z))
= \alpha^2 f(x)f(y)f(z).$
Since $f\neq0$, then we deduce that
$f(yz)= \alpha f(y)f(z)$ for all $y,z$, which implies that    $\alpha f:=\chi $, where $\chi$ is an exponential.  By using again (\ref{333})   we get
$\chi(x)\chi\circ\sigma(y)=\chi(x\sigma(y))=\chi(x)\chi(y),$
so $\chi\circ \sigma =\chi$.
	 	Now, by substituting  $f=\dfrac{\chi}{\alpha}$  into (\ref{04})  we obtain
	 		$
	 		\dfrac{1}{\alpha}\chi(x)\chi(y) \int_{S}\chi(t)d\mu(t)
	 		= \dfrac{\chi(x) }{\alpha}\dfrac{\chi(y) }{\alpha},$ which implies that  $\dfrac{1}{\alpha}=\int_{S}\chi(t)d\mu(t)$. Finally,  $f=[\int_{S}\chi(t)d\mu(t)] \chi$.
	 		The converse is straightforward verification. This completes the proof.
	 	\end{proof}
		
			In the following we prove some useful lemmas which will be used  in the proof of our first main result (Theorem 3.6).
			
		  \begin{lem}
		 						  	\label{uu}
		 						 Let $f, g :S\longmapsto\mathbb{C}$ be a solution  of  equation (\ref{6}) such that  $f,g$ are linearly independent. Then $\int_{S}f(t)d\mu(t)=0$.
\end{lem}
\begin{proof}
Suppose that $\int_{S}f(t)d\mu(t)\neq0$. By replacing $y$ by $y\sigma(s)k$  in (\ref{6}) and integrating the result obtained with respect to $s$ and $k$  we obtain
\begin{equation*}\begin{split}&
\int_{S}\int_{S}\int_{S}f(x\sigma(y\sigma(s)k)t)d\mu(t)d\mu(k)d\mu(s)\\	&=f(x)\int_{S}\int_{S}g(y\sigma(s)k)d\mu(k)d\mu(s)-g(x)\int_{S}\int_{S}f(y\sigma(s)k))
d\mu(k)d\mu(s)\\
&=\int_{S}\int_{S}\int_{S}f(x\sigma(y\sigma(s))\sigma(k)t)d\mu(k)d\mu(s)d\mu(t)\\
&=\int_{S}f(x\sigma(y\sigma(s)))d\mu(s)\int_{S}g(k)d\mu(k)-\int_{S}f(k)d\mu(k)\int_{S}g(x\sigma(y\sigma(s)))d\mu(s).
\end{split}\end{equation*} Then,
		 						  		\begin{equation*}\begin{split}&\int_{S}g(k)d\mu(k)[f(x)g(y)- f(y)g(x)]-\int_{S}f(k)d\mu(k)\int_{S}g(x\sigma(y )s)d\mu(s)\\
&=f(x)\int_{S}\int_{S}g(y\sigma(s)k)d\mu(k)d\mu(s)\\
&\quad-g(x)\left[ f(y)\int_{S}g(s)d\mu(s)-g(y)\int_{S}f(s)d\mu(s)\right] ,\end{split}\end{equation*}
		 						  		which can rewritten as follows
		 						  		\begin{equation}
		 						  		\label{n3}
		 						  		\int_{S}f(t)d\mu(t)\left[ \int_{S}g(x\sigma(y )k)d\mu(k)+g(x)g(y)\right] \end{equation}$$=f(x)\left[ -\int_{S}\int_{S}g(y\sigma(k)t)d\mu(k)d\mu(t)+g(y)\int_{S}g(t)d\mu(t)\right] .
		 						  		$$
		 						  	By assumption  $\int_{S}f(t)d\mu(t)\neq0$ then, from equation (\ref{n3})  we get
		 						  		\begin{equation}
		 						  		\label{n4}
\int_{S}g(x\sigma(y )t)d\mu(t)= f(x)h(y)-g(x)g(y),
\end{equation}
where
\begin{equation*}
h(y):= -\dfrac {\int_{S}\int_{S}g(y\sigma(k )s)d\mu(k)d\mu(s)-g(y)\int_{S}g(t)d\mu(t)} {\int_{S}f(t)d\mu(t)}.
\end{equation*}
Now, by substituting (\ref{n4}) into (\ref{n3})  we obtain
\begin{equation*}
f(x)h(y)\int_{S}f(t)d\mu(t)	=f(x)\left[ -\int_{S}h(t)d\mu(t)f(y)+2g(y)\int_{S}g(t)d\mu(t)\right].
\end{equation*}
Since  $\int_{S}f(t)d\mu(t)\neq0$ and $ f\neq0$ we derive from the identity above that
$h=\alpha f+\beta g$,
where $\alpha:=-\dfrac{\int_{S}h(t)d\mu(t)}{\int_{S}f(t)d\mu(t) }$ \;and\; $\beta:=\dfrac{2\int_{S}g(t)d\mu(t)}{\int_{S}f(t)d\mu(t) }$.
Then $$\alpha =-\dfrac{\int_{S}h(t)d\mu(t)}{\int_{S}f(t)d\mu(t) }=-\dfrac{\alpha \int_{S}f(t)d\mu(t)+\beta \int_{S}g(t)d\mu(t)}{\int_{S}f(t)d\mu(t) }=-\alpha-\dfrac{\beta^2}{2}.$$ So, $\alpha=- {\beta^2}/{4}$ and		 						   	$h(y)=-\dfrac{\beta^2}{4} f(y)+\beta g(y).$ \\
Now, equation (\ref{n4}) can be written as  follows
\begin{equation}
\label{132}
\int_{S}g(x\sigma(y )t)d\mu(t)= -g(x)g(y)-\dfrac{\beta^2}{4} f(x)f(y)+\beta f(x)g(y) .
\end{equation}
From (\ref{6}) and (\ref{132}) we obtain
		 						  			 	\begin{equation}
		 						  			 	\label{n5}
		 						  			 	  \int_{S}( g-\dfrac{\beta}{2} f) (x\sigma(y)t)d\mu(t) =-(g-\dfrac{\beta}{2} f) (x ) (g-\dfrac{\beta}{2} f)(y).		 						  			 	\end{equation} for all $x,y\in S.$
By applying Proposition \ref{p1} to (\ref{n5}), and taking into account that $f,g$ is linearly independent, we derive that there exists  an exponential   $\chi$ on $S$ such that  $\chi\circ \sigma=\chi$ and  $ g-\dfrac{\beta}{2} f=-\big[\int_{S}\chi(t)d\mu(t)\big]\chi$.  \\
		 						  			 	    By putting $x=s$ and $y=t$ in (\ref{n3}) and integrating the result obtained with respect to $s$ and $r$ we get
		 						  			 	   \begin{equation}
		 						  			 	   \label{ss}
		 						  			 	   \int_{S}\int_{S}\int_{S}g(s\sigma(t)k)d\mu(s)d\mu(k)d\mu(t)=0.
		 						  			 	   \end{equation}   By replacing $x$ by $k$ and $y$ by $s$ in (\ref{6}) and integrating the result obtained with respect to $k$ and $s$ we obtain
		 						  			 	   \begin{equation}
		 						  			 	   \label{24}  \int_{S}\int_{S}\int_{S}f(k\sigma(s)t)d\mu(k)d\mu(s)d\mu(t) \end{equation}$$= \int_{S}f(k)d\mu(k)\int_{S}g(s)d\mu(s)-\int_{S}f(s)d\mu(s)\int_{S}g(k)d\mu(k)=0.
		 						  			 	   $$ Now,  putting $x=k$ and $s=y$ in (\ref{n5}), integrating the result obtained with respect to $k$ and $s$  and using (\ref{ss})  and (\ref{24}) we obtain
		 						  			 	   $$  \int_{S} \int_{S}\int_{S}(g-\dfrac{\beta}{2} f)(k\sigma(s)t)d\mu(s)d\mu(k)d\mu(t)d\mu(s)=-\left( \int_{S}\chi(s)d\mu(s)\right) ^4=0,$$
which contradicts that $\int_{S}\chi(s)d\mu(s)\neq0$.  Thus     $\int_{S}f(t)d\mu(t)=0$.
\end{proof}
\begin{rem}
If $S$ is a monoid and $f,g$ is a solution of (\ref{6}) we easily obtain $\int_{S}f(t)d\mu(t)=0$. In fact, by putting  $x=y=e$ in (\ref{6}) we get $\int_{S}f(t)d\mu(t)= f(e)g(e)-f(e)g(e)=0$.
		 						  	\end{rem}
		 						 \begin{lem}
		 						 		\label{r8}
		 						 		Let    $f, g :S\longmapsto\mathbb{C}$ be a solution  of the functional equation (\ref{6}) such that $f$ and $g$ are linearly independent. Then  we have
		 						 		\begin{equation}
		 						 		\label{08}
		 						 	f(x\sigma(y))\int_{S}\int_{S}g(\sigma(s)k)d\mu(k)d\mu(s)	\end{equation}$$= \big[f(x)g(y)-f(y)g(x)\big]\int_{S}g(k)d\mu(k)+g(x\sigma(y))\int_{S}\int_{S}f(\sigma(s)k)d\mu(k)d\mu(s)$$
		 						 for all $x,y\in S.$	
		 						 \end{lem}
		 						 \begin{proof}
		 					 	Replacing  $x$ by $x\sigma(y\sigma(k))$ and $y$ by $s$ in (\ref{6}) and integrating the result obtained with respect to $k$ and $s$ we get
\begin{equation*}\begin{split}&		 						  \int_{S}\int_{S}\int_{S}f(x\sigma(y\sigma(k))\sigma(s)t)d\mu(t)d\mu(s)d\mu(k)\\  &=\int_{S}f(x\sigma(y)k))d\mu(k)\int_{S}g( s)d\mu(s)-\int_{S}f(s)d\mu(s)\int_{S}g(x\sigma(y)k)d\mu(k)\\
&=\int_{S}\int_{S}\int_{S}f(x\sigma(y)\sigma(\sigma(k)s)t)d\mu(t)d\mu(s)d\mu(k)\\
&=f(x\sigma(y))\int_{S}\int_{S}g(\sigma(k)s)d\mu(s)d\mu(k)-g(x\sigma(y))\int_{S}\int_{S}f(\sigma(k)s)d\mu(s)d\mu(k).\end{split}\end{equation*}
Since, from Lemma \ref{uu}: $\int_{S}f(s)d\mu(s)=0$, then we get the desired result.
\end{proof}	  	
\begin{lem}
\label{r10}
Let $f, g :S\longmapsto\mathbb{C}$ be a solution  of the functional equation (\ref{6}) such that $\{f, g\}$ is linearly independent. Then we have the following. \\
		 						 	(i)  $\int_{S}g(s)d\mu(s)\neq0$. \\
		 						 	(ii ) If $\int_{S}\int_{S}g(\sigma(s)k)d\mu(s)d\mu(k)=0
		 						 	$ \text{then} $ \int_{S}\int_{S}f(\sigma(s)k)d\mu(s)d\mu(k)\neq0$.
\end{lem}
\begin{proof} (i) In view of  Lemma \ref{uu} we have $\int_{S}f(s)d\mu(s)=0$.
Assume  that $\int_{S}g(s)d\mu(s)=0$. Making the substitution $(x,yk)$  in (\ref{6}) and integrating the result obtained with respect to $k$ we get  that
\begin{equation*}\begin{split}&			    		\int_{S}\int_{S}f(x\sigma(yk))t)d\mu(k)d\mu(t)\\
&=f(x )\int_{S}g(yk)d\mu(k)-g(x)\int_{S}f(yk)d\mu(k) \\
&=\int_{S}\int_{S}f(x\sigma(y)\sigma(k))t)d\mu(k)d\mu(t)\\
&=f(x\sigma(y))\int_{S}g(k)d\mu(k)-g(x\sigma(y))\int_{S}f(k)d\mu(k)=0,
\end{split}\end{equation*}
which implies that $f(x )\int_{S}g(yk)d\mu(k)-g(x)\int_{S}f(yk)d\mu(k)=0,  x,y\in S.$ As $\{f, g\}$ is linearly independent, we get  that  $\int_{S}f(yk)d\mu(k)=\int_{S}g(yk)d\mu(k)=0$ for all $y\in S$. Combining this with (\ref{6}) we obtain  $f(x)g(y)-f(y)g(x)=0$ for all $x,y\in S.$ This contradicts  the fact that $\{f, g\}$ is linearly independent. So, we have (i).\\
		 						    	(ii)  Assume  that  $\int_{S}\int_{S}f(\sigma(s)k)d\mu(s)d\mu(k)=0$. \\Since $\int_{S}\int_{S}g(\sigma(s)k)d\mu(s)d\mu(k)=0$ by hypothesis, then  equation (\ref{08})  implies
		 						    	$
		 						    	\int_{S}g(t)d\mu(t) [f(x)g(y)-f(y)g(x)]=0$ for all $x,y\in S.
		 						    	$
Since $\int_{S}g(t)d\mu(t)\neq0$, we get that
		 						    	$
		 						    	f(x)g(y)=  f(y)g(x)
		 						    	$ for all $x,y\in S$.
This contradicts that $f $ and $g$ are linearly independent. The Lemma \ref{r10} is proved.
\end{proof}		 						   	
Now we are ready to find  the solutions of the integral Kannappan-Sine subtraction law (\ref{6}) on semigroups. The following theorem is the first main result.
\begin{thm}\label{t01}
The solutions $f,g :S\longrightarrow \mathbb{C}$ of    integral Kannappan-Sine subtraction law (\ref{6}) can be listed as follows.\\
(1) $f=0$ and $g$ arbitrary.\\
(2)  $f$ is any non-zero function such that $\int_{S}f(xyt)d\mu(t)=0$ for all $x,y\in S$  and $g=k f$, where $k\in \mathbb{C}$ is a constant.\\
(3) There exist  constants $\gamma, b \in \mathbb{C}^{*}, c  \in \mathbb{C} \setminus \{\pm 1\} $ and an exponential  $\chi$ on $S$  with $\chi \neq \chi\circ\sigma$, $\int_{S}\chi(t)d\mu(t)=\dfrac{-2b}{1+c}$ and $\int_{S}\chi\circ\sigma(t)d\mu(t)=\dfrac{  2b}{1-c}$ such that
\begin{equation*}
f=  \dfrac{\chi+\chi\circ\sigma}{2\gamma} +c\dfrac{\chi-\chi\circ\sigma}{2\gamma}   \; and \;  g=b (\chi-\chi\circ\sigma ).
\end{equation*}
(4) There exist  constants $\beta, b \in \mathbb{C}^{*}, c  \in \mathbb{C}   $ and an exponential  $\chi$ on
		 						   	$S$ with
		 						   	$\chi \neq \chi\circ\sigma$ and $\int_{S}\chi(t)d\mu(t)=  \int_{S}\chi\circ\sigma(t)d\mu(t)=1/\beta$ such that
		 						   	\begin{equation*}
		 						   	f=b
		 						    (\chi-\chi\circ\sigma) \hspace{0.4cm} and \hspace{0.4cm}  g=  \dfrac{\chi+\chi\circ\sigma}{2\beta}+c\dfrac{\chi-\chi\circ\sigma}{2\beta}.
		 						   	\end{equation*}
		 						   	(5) There exist  constants $\alpha,\delta, b  \in \mathbb{C}^{*}, c\in \mathbb{C}  $    and an exponential
		 						   		$\chi$  on	$S$
		 						   		with  $ \alpha \neq\pm (2b \delta +\alpha c)$,  $\chi \neq \chi\circ\sigma$, $	\int_{S}\chi (t)d\mu(t)=	\dfrac {2b }{\alpha (1+c)+2b\delta   } $ and\\ $ 	 \int_{S}\chi\circ \sigma(t)d\mu(t)=\dfrac {2b }{\alpha (c-1)+2b\delta   }$
		 						   	 such that
		 						   	\begin{equation*}
		 						    f= \alpha   \dfrac{\chi+\chi\circ\sigma}{2\delta }+  (\alpha c+2b\delta) \dfrac{\chi-\chi\circ\sigma}{2\delta }   \hspace{0.2cm} and \hspace{0.2cm}  g=   \dfrac{\chi+\chi\circ\sigma}{2\delta} +c  \dfrac{\chi-\chi\circ\sigma}{2\delta}.
\end{equation*}
(6) There exist a constant $\gamma\in \mathbb{C}^{*} $, an exponential  $\chi$ on $S$ and a non-zero
function $\phi_{\chi}$: $S\longrightarrow \mathbb{C}$ with $\chi\circ\sigma=\chi, \phi_\chi\circ\sigma=$
$-\phi_\chi$,		 						   	$\int_S\chi(t)d\mu(t)\neq 0$ and $\int_S \phi_\chi (t)d\mu(t)=(\int_S \chi(t)d\mu(t))^2$ such that
$
		 						   	f=\dfrac{1}{\gamma} \left(  {\chi} -\dfrac{1}{\int_S \chi(t)d\mu(t)} \phi_{\chi}\right) $ and $ g= \phi_{\chi}
		 						 $.\\
		 						   	(7) There exist  constants $ \alpha,  c \in \mathbb{C}$ and $ \delta \in \mathbb{C}^{*}$, an exponential  $\chi$ on $S$,  a non-zero
		 						   	function $\phi_{\chi}$: $S\longrightarrow \mathbb{C}$
		 						    with $\alpha c+\delta\neq0$, $\chi\circ\sigma$
		 						   $=\chi$,  $\phi_\chi \circ\sigma=-\phi_\chi,$
		 						   	$\int_S \chi(t)d\mu(t)=\dfrac{1}{\alpha c+\delta} $ and $\int_S \phi_\chi (t)d\mu(t)=\dfrac{-\alpha}{(\alpha c+\delta)^2} $ such that
		 						  $$
		 						   	f=   \dfrac{1}{\delta} [  \alpha\chi+ (\alpha c +\delta) \phi_{\chi}]\;\;and \;\;  g=   \dfrac{1}{\delta} \left( \chi+ c \phi_{\chi}\right).
		 						   $$\\(8) There exist a constant $  c\in \mathbb{C} $, an exponential  $\chi$ on $S$, a non-zero
		 						   function $\phi_{\chi}$: $S\longrightarrow \mathbb{C}$ with $\chi\circ\sigma=\chi, \phi_\chi\circ\sigma=$
		 						   $-\phi_\chi, $ and
		 						   $\int_S\chi(t)d\mu(t)\neq 0$, $\int_S \phi_\chi (t)d\mu(t)=0$ and
		 						   $
		 						   f=  \phi_{\chi},$  $ g= \int_S \chi(t)d\mu(t)(\chi+c\phi_{\chi})
		 						   $.\\
		 						   	Moreover if $S$ is a topological semigroup and $f,g \in C(S)$, then $ \chi, \chi\circ\sigma \in C(S)$ and  $\phi_\chi\in C(S)$.
		 						   	 \end{thm}
		 							 \begin{proof}
		 							 	 If $f=0$ then $g$ can be chosen arbitrary. The solution occurs in part (1). For the rest of the proof we assume that  $f\neq0$.
		 							 	 If $f$ and $g$ are linearly dependent, then there exists a constant $k \in \mathbb{C}$ such that $g=k f$ and then Eq. (\ref{6}) reduces to $\int_{S}f(xyt)d\mu(t)=0$ for all $x,y \in S$.  This proves the solution family (2). From now on we assume that $f$ and $g$ are linearly independent. By Lemma \ref{r10} (i) we have $\int_{S}g(t)d\mu(t)\neq0$. We will discuss two cases:\\{Case 1}: Suppose $\int_{S}\int_{S}g(\sigma(s)k )d\mu(s)d\mu(k)=0$, then from Lemma \ref{r10} (ii) we have  $\int_{S}\int_{S}f(\sigma(s)k )d\mu(s)d\mu(k)\neq0$ and therefore Eq.(\ref{08}) becomes
\begin{equation}
\label{r88}
g(x\sigma(y))= \gamma g(x)f(y)- \gamma f(x)g(y),  \; x,y \in S,
\end{equation}
where $\gamma:=\dfrac{\int_{S}g(k)d\mu(k)}{\int_{S}\int_{S}f(\sigma(s)k )d\mu(s)d\mu(k)} \neq0$. From  (\ref{r88}) we read that the pair $(g, \gamma f)$ satisfies the sine subtraction law (\ref{7}), so according to \cite[Theorem 4.2]{vb}  and taking into account that $f$ and $g$ are linearly independent,  we infer that we have only  the last possibilities (c) and (d) of \cite[Theorem 4.2]{vb}:
		 							 	\\(i)  There exist constants $b\in \mathbb{C}^{*}, c \in \mathbb{C}$ and  an exponential  $\chi $ on $S$, with $\chi \circ \sigma \neq \chi$  such that $g= b(\chi-\chi\circ\sigma)$ \;\;and \;\; $\gamma f=\dfrac{\chi +\chi\circ\sigma}{2 }+c  \dfrac{\chi-\chi\circ\sigma}{2 }.$
Substituting this in (\ref{6}) we get after some rearrangement
$
\left( 	(1+c)\int_S\chi(t)d\mu(t)+2 b\right) \chi(xy)+\left( (1-c )\int_{S}\chi\circ\sigma(t)d\mu(t)-2 b\right) \chi\circ \sigma(xy)=0		 							 	 	$ for all $ x, y \in S.$
From \cite[Theorem 3.18 ]{g} we deduce that		 							 	 		\begin{equation}
\label{rrr}		 							 	 	  	(1+c)\int_{S}\chi(t)d\mu(t)=-2 b\;\;\;  and\; \;\;  	(1-c )\int_{S}\chi\circ\sigma(t)d\mu(t)=2 b.  \end{equation}
As $b \neq0$ we get from (\ref{rrr}) that $c \neq \pm1$. So,		 							 	 	  		\begin{equation*}
		 							 	 	  	 \int_{S}\chi(t)d\mu(t)=\dfrac{ -2b}{1+c }\;\;\;  and \;\;\;    \int_{S}\chi\circ\sigma(t)d\mu(t)=\dfrac{ 2b}{1-c }.  \end{equation*}
The solution occurs in part (3).
\\(ii) There exist a  constant $c\in \mathbb{C}$, an exponential  $\chi$ on $S$ and a function $\phi_{\chi}$ satisfying (\ref{111}) such that $\chi=\chi\circ\sigma$, $\phi_{\chi}\circ\sigma=-\phi_{\chi}\neq0$,
		 							 	 	  	 $g=   \phi_{\chi} \;and \hspace{0.4cm}  \gamma f= \chi+c\phi_{\chi}.$
		 							 	 	  	  Furthermore, $-c\int_{S}\phi_{\chi}(t)d\mu(t)=\int_{S}{\chi}(t)d\mu(t)$ and $\int_{S}\phi_{\chi}(t)d\mu(t)=\int_{S}g(t)d\mu(t)\neq0$ because $\int_{S}f(t)d\mu(t)=0$ by Lemma \ref{uu}. By substituting the new expression of $f$ and $g$ into (\ref{6}) and using that the set $\{\chi, \phi_{\chi}\}$ is linearly independent \cite[Lemma 5.1(b)]{d} we get that $\int_S \chi(t)d\mu(t)\neq 0$,  $c\neq0$ and $c=\dfrac{-1}{\int_{S}\chi(t)d\mu(t)}$ and then $\int_{S}\phi_{\chi}(t)d\mu(t)=(\int_{S}{\chi}(t)d\mu(t))^{2}$. Then  we conclude that we deal with case   (6).\\
		 							 	 	{Case 2}: Suppose $\int_{S}\int_{S}g(\sigma(s)k )d\mu(s)d\mu(k)\neq0$.\\
		 							 	 	 \underline{Subcase 2.1}. $\int_{S}\int_{S}f(\sigma(s)k )d\mu(s)d\mu(k)=0$. Then (\ref{08}) reduces to
		 							 	 	 \begin{equation}
		 							 	 	 \label{rtt}
		 							 	 	 f(x\sigma(y))= \beta f(x)g(y)-\beta f(y)g(x),  \; x,y \in S,
		 							 	 	 \end{equation}
		 							 	 	 where $\beta:=\dfrac{\int_{S}g(k)d\mu(k)}{\int_{S}\int_{S}g(\sigma(s)k )d\mu(s)d\mu(k)} \neq0$. From  (\ref{rtt}) we read that the pair $(f, \beta g)$ satisfies the sine subtraction law (\ref{7}), so according to \cite[Theorem 4.2]{vb}  and taking into account that $f$ and $g$ are linearly independent,  we infer that we haveonly  the following possibilities:\\		 							 	 	
(i)  There exist constants $b\in \mathbb{C}^{*}, c \in \mathbb{C}$ and  an exponential  $\chi $ on $S$ with $\chi \circ \sigma \neq \chi$   such that $f= b(\chi-\chi\circ\sigma)$ and $ \beta g=\dfrac{\chi +\chi\circ\sigma}{2 }+c  \dfrac{\chi-\chi\circ\sigma}{2 }.$
		 							 	 	 By substituting  the new expression of $f$ and $g$ into (\ref{6}) we get  after simplification that
	$$\chi(x\sigma(y))\left[ b \int_S \chi(t)d\mu(t)-\dfrac{b}{\beta} \right] +\chi(\sigma(x)y)\left[ -b \int_S \chi(\sigma(t))d\mu(t)+\dfrac{b}{\beta} \right] =0$$ for all $x,y\in S,$
		 							 	 	 which implies, by \cite[Theorem 3.18]{g}, that
		 							 	 	 \begin{equation*}
		 							 	 	  \beta\int_{S}\chi\circ \sigma(t)d\mu(t)-1=0\hspace{0.3cm}  and \hspace{0.3cm}  1-	 \beta\int_{S}\chi (t)d\mu(t) =0.  \end{equation*}
		 							 	 	This gives $\int_{S}\chi\circ \sigma(t)d\mu(t)= \int_{S}\chi(t)d\mu(t)=1/\beta $, and  we deal with case   (4).\\		 							 	 	
(ii) There exist a  constant $c\in \mathbb{C}$, an exponential  $\chi$ on $S$, and $\phi_{\chi}$ solution of (\ref{111}) such that $\chi=\chi\circ\sigma$, $\phi_{\chi}\circ\sigma=-\phi_{\chi}\neq0$, and
		 							 	 	$f=   \phi_{\chi} \;and  \;   \beta g= \chi+c\phi_{\chi}.$
		 							 	 	 Furthermore,  since by Lemma \ref{uu}: $\int_{S}f(t)d\mu(t)=0$ then $\int_{S}\phi_{\chi}(t)d\mu(t)=0$, and by inserting these forms of $f,g$ into  (\ref{6}) and using  \cite[Lemma 5.1(b)]{d}
		 							 	 	  we get that  $\dfrac{1}{\beta}=\int_{S}{\chi}(t)d\mu(t)\neq0$. Then  we deal with case (8). \\
		 							 	 	 \underline{Subcase 2.2}: If $\int_{S}\int_{S}f(\sigma(s)k )d\mu(s)d\mu(t)\neq0$  then Eq. (\ref{08}) becomes
		 							 	 	 	\begin{equation}
		 							 	 	 	\label{rvv}
		 							 	 	 	 f(x\sigma(y))= \delta f(x)g(y)-\delta f(y)g(x)+\alpha   g(x\sigma(y)),  \; x,y \in S,
		 							 	 	 	\end{equation}
		 							 	 	 where $\delta:=\dfrac{\int_{S}g(k)d\mu(k)}{\int_{S}\int_{S}g(\sigma(s)k )d\mu(s)d\mu(k)} \neq0$ (see Lemma \ref{r10}(i))  and \\$ \alpha:=\dfrac{\int_{S}\int_{S}f(\sigma(s)k )d\mu(s)d\mu(k)}{\int_{S}\int_{S}g(\sigma(s)k )d\mu(s)d\mu(k)} \neq0.$\\
		 							 	 	 Now, we can reformulate  Eq. (\ref{rvv})  as follows
		 							 	 	 $$
		 							 	 	 \left(  f-\alpha g\right)(x\sigma(y)))=\delta \left(  f-\alpha g\right)(x)g(y)-\delta  \left(  f-\alpha g\right)(y) g(x),\; x, y \in S.
		 							 	 	$$
		 					Then, the pair $ \left(   f-\alpha g, \delta g\right) $ satisfies the sine subtraction law (\ref{7}).  So, in view of \cite[Theorem 4.2]{vb} and using the assumption that $\{f,g\}$ is   independent, we infer that we have only the following two possibilities: \\ (i)  There exist constants $b\in \mathbb{C}^{*}, c \in \mathbb{C}$ and  an exponential  $\chi $ on $S$, with $\chi \circ \sigma \neq \chi$ such that
		 						 \begin{equation*}
		 						 f-\alpha g=b(\chi -\chi\circ \sigma)\;\;\; and \;\;\; \delta g= \dfrac{\chi +\chi\circ \sigma}{2}+c\dfrac{\chi -\chi\circ \sigma}{2},
		 						 \end{equation*}
		 						 which implies that
		 						 \begin{equation*}
		 						  f= \alpha   \dfrac{\chi+\chi\circ\sigma}{2\delta }+  (\alpha c+2b\delta) \dfrac{\chi-\chi\circ\sigma}{2\delta }   \hspace{0.4cm}and \hspace{0.4cm}    g=\dfrac{\chi+\chi\circ\sigma}{2 \delta}+c  \dfrac{\chi-\chi\circ\sigma}{2\delta }.
		 						 \end{equation*}  Combining this with (\ref{6})  we  find   after some computations that
		 						 $$
		 						 \left( 2 b-(\alpha+\alpha c +2b \delta)\int_S\chi (t)d\mu(t) \right) \chi(x\sigma(y)) $$$$-\left( 	2 b+(\alpha-\alpha c -2b \delta) \int_{S}\chi(\sigma(t)d\mu(t)\right) \chi( \sigma(x) y)=0
		 						$$
		 						for all $x,y \in S$. Now, from \cite[Theorem 3.18]{g} we deduce that
		 						 \begin{equation}
		 						 \label{n1}
		 						 	2 b-(\alpha+\alpha c +2b \delta)\int_{S}\chi (t)d\mu(t) =0\;  and\;  		2 b+(\alpha-\alpha c -2b \delta) \int_{S}\chi\circ \sigma(t)d\mu(t)=0.  \end{equation}
		 						Furthermore, we have $\alpha \neq\pm (2b \delta +\alpha c)$,  since $b\neq 0$. Finally, we have
		 					$	
		 				\int_{S}	\chi (t)d\mu(t)=	\dfrac {2b }{\alpha (1+c)+2b\delta   }$ and $ \int_{S}\chi\circ \sigma(t)d\mu(t)=\dfrac {2b }{\alpha (c-1)+2b\delta   }.
		 				$ We see that we deal with case (5).\\(ii) There exist a  constant $c\in \mathbb{C}$, an exponential  $\chi $ on $S$, and a non-zero function $\phi_{\chi }$ solution  (\ref{111}) such that
		 							$
		 							   f-\alpha g= \phi_{\chi }\; and \;   \delta g=\chi +c   \phi_{\chi }
		 							 $
		 							 with $\chi  \circ \sigma = \chi $, $\phi_{\chi }\circ \sigma = -\phi_{\chi }$.
		 							 So, we have
		 							 $f =\dfrac{1}{\delta}[ \alpha \chi + (\alpha c+\delta)\phi_{\chi }]\;\;\; and \;\;\;    g=\dfrac{1}{\delta}(\chi +c   \phi_{\chi }).
		 							 $ Now,  by substituting the new expression of $f$ and $g$ into (\ref{6}), using \cite[Lemma 5.1(b)]{d}, and that $\int_S f(t)d\mu(t)=0$  by Lemma \ref{uu}  we conclude that $\alpha c+\delta\neq0$,
		 							   $\int_{S}\chi(t)d\mu(t)=\dfrac{1}{\alpha c+\delta}$ and $\int_S \phi_{\chi }d\mu(t)=\dfrac{-\alpha}{(\alpha c+\delta)^2} $.  We conclude that we deal with case (7).
		 							
		 							   	Conversely,  simple computations prove that the formulas above for      $f$ and $g$   define  solutions of (\ref{6}).
		 							   	
		 							  Finally,  suppose that $S$ is topological semigroup and $f, g \in C(S)$. The  continuity of cases (1) and (2) are evident.  The continuity  of $\chi$ and $\chi \circ \sigma$ in case (3)  follows directly from the form of $g$ by applying \cite[Theorem 3.18(d)]{g} since $b\neq0$ and  $\chi\neq \chi \circ \sigma$. The case (4)   can be treated as the case (3). For the case (5) we have $f-\alpha g= b (\chi -\chi\circ \sigma).
$ Then the continuity of $\chi$ and $\chi \circ \sigma$ follows from \cite[Theorem 3.18]{g} because $ b\neq0$ and $\chi\neq \chi \circ \sigma$. For the case (6)  we have $g=\phi_{\chi}$ is  continuous. As $\gamma f +\dfrac{1}{  \int_S \chi(t)d\mu(t)} \phi_{\chi}=\chi$ we get $\chi \in C(S)$. The parts (7) and (8) can be treated as case (6). This completes the proof of Theorem \ref{t01}.
\end{proof}
\section{Solution of  equation (\ref{elqorachi2})}
We start with three  lemmas which contain some useful results about the solutions of the integral Kannappan-Sine addition law (\ref{elqorachi2}).
		 							 	 \begin{lem}
		 							 	 	\label{l1}
		 							 	 	Let    $f, g :S\longmapsto\mathbb{C}$ be a solution  of  equation (\ref{elqorachi2}). Then we have the following.\\
		 							 	 	(i) If $\int_{S}f(t)d\mu(t)=0 $, then  for all $ x,y \in S $
		 							 	 	\begin{equation}
		 							 	 	\label{14}
		 							 	 	f(x\sigma(y)) \int_{S} \int_{S}  	   g(\sigma(t)s)d\mu(s) d\mu(t) \end{equation}   $$   = \left[ f(x)g(y)+f(y)g(x)\right] \int_{S}  	   g(s)d\mu(s)-g(x\sigma(y))\int_{S}\int_{S}    f(\sigma(t)s)d\mu(s)  d\mu(t) .$$
		 							 	 	(ii) If $\int_{S} f(t)d\mu(t)\neq0$,  then there exists a constant $\alpha \in \mathbb{C}$ such that
		 							 	 	\begin{equation}
		 							 	 	\label{63}
		 							 	 	\int_{S}	g(x\sigma(y) t)d\mu(t)=g(x)g(y)+\alpha^2 f(x)f(y), \;\; x,y \in S.
		 							 	 	\end{equation}
		 							 	 \end{lem}
		 							 	 \begin{proof}
		 							 	 	(i) Suppose that   $ \int_{S} f(t)d\mu(t)=0 $. 	Replacing  $x$ by $x\sigma(y\sigma(k))$ and $y$ by $s$ in (\ref{elqorachi2}) and integrating the result obtained with respect to $k$ and $s$ we get
		 							 	 	\begin{equation*}\begin{split}&		 							 	 	\int_{S}\int_{S}\int_{S}f(x\sigma(y\sigma(k))\sigma(s)t)d\mu(t)d\mu(s)d\mu(k)\\  &=\int_{S}f(x\sigma(y)k))d\mu(k)\int_{S}g( s)d\mu(s)+\int_{S}f(s)d\mu(s)\int_{S}g(x\sigma(y)k)d\mu(k)\\
&=\int_{S}f(x\sigma(y)k))d\mu(k)\int_{S}g(s)d\mu(s)\\
&=\int_{S}\int_{S}\int_{S}f(x\sigma(y)\sigma(\sigma(k)s)t)d\mu(t)d\mu(s)d\mu(k)\\
&=f(x\sigma(y))\int_{S}\int_{S}g(\sigma(k)s)d\mu(s)d\mu(k)+g(x\sigma(y))\int_{S}\int_{S}f(\sigma(k)s)d\mu(s)d\mu(k).
\end{split}\end{equation*}	
This proves (\ref{14}).\\

		 							 	 		 							 	 	(ii) By replacing $y$ by $y\sigma(s)k$  in (\ref{elqorachi2}) and integrating the result obtained with respect to $s$ and $k$  we obtain
		 							 	 	\begin{equation*}\begin{split}&
\int_{S}\int_{S}\int_{S}f(x\sigma(y\sigma(s)k)t)d\mu(t)d\mu(k)d\mu(s)\\	&= f(x)\int_{S}\int_{S}g(y\sigma(s) k)d\mu(k)d\mu(s)+g(x)\int_{S}\int_{S}f(y\sigma(s)k))d\mu(k)d\mu(s)\\
&=\int_{S}\int_{S}\int_{S}f(x\sigma(y\sigma(s))\sigma(k)t)d\mu(k)d\mu(s)d\mu(t)\\
&=\int_{S}f(x\sigma(y\sigma(s)))d\mu(s)\int_{S}g(k)d\mu(k)+\int_{S}f(k)d\mu(k)\int_{S}g(x\sigma(y\sigma(s)))d\mu(s).
\end{split}\end{equation*}
Then, we deduce that
		 							 	 	\begin{equation*}\begin{split}&\int_{S}g(k)d\mu(k)[f(x)g(y)+ f(y)g(x)]+\int_{S}f(k)d\mu(k)\int_{S}g(x\sigma(y )s)d\mu(s)\\
&=f(x)\int_{S}\int_{S}g(y\sigma(s)k)d\mu(k)d\mu(s)\\
&\quad+g(x)\left[f(y)\int_{S}g(s)d\mu(s)+g(y)\int_{S}f(s)d\mu(s)\right] ,\end{split}\end{equation*}
		 							 		from which we   derive
		 $$\int_{S}f(s)d\mu(s)\left[  \int_{S} g(x\sigma(y)t)d\mu(t)-g(x)g(y)\right]   $$
			$$ = f(x)\left[\int_{S} \int_{S} 	   g(y\sigma(s)t)d\mu(s) d\mu(t)-g(y)\int_{S}g(s)d\mu(s) \right].  $$
		 	Since $\int_{S}  f(s)d\mu(s) \neq0$ by assumption, we get that
		 							 	 	\begin{equation}
		 							 	 	\label{vvv} \int_{S} g(x\sigma(y) t)d\mu(t)-g(x)g(y)=f(x)\psi(y),
		 							 	 	\end{equation}
		 							 	 	where
		 							 	 	\begin{equation}
		 							 	 	\label{uuu}
		 							 	 	\psi(y):=\dfrac{\int_{S} \int_{S} 	   g(y\sigma(s)t)d\mu(s) d\mu(t)-g(y)\int_{S}g(s)d\mu(s)}{\int_{S}  f(s)d\mu(s)}.
		 							 	 	\end{equation}
		 							 	 	By replacing $y$ by $s$    in (\ref{vvv}) and integrating the result obtained with respect to  $s$  and using (\ref{uuu}), we obtain
		 							 	 	\begin{equation*}\begin{split}		 							 	 		f(x)\int_{S}\psi(s)d\mu(s) &=\int_{S}\int_{S} g(x\sigma(s) t)d\mu(s)d\mu(t)-g(x)\int_{S}g(s)d\mu(s)\\		 							 	 		&=\psi(x)\int_{S}  f(s)d\mu(s).
\end{split}\end{equation*}
Since $\int_{S}  f(s)d\mu(s) \neq0$, we deduce that  $\psi(x)= \xi f(x)$, where $\xi:=\dfrac{\int_{S}\psi(s)d\mu(s)}{\int_{S}  f(s)d\mu(s)}. $
		 							 	 	Choosing $  \alpha \in \mathbb{C}$ such that  $ \alpha^2=\xi$, equation  (\ref{vvv}) reads
		 							 	 	$
		 							 	 	\int_{S} g(x\sigma(y) t)d\mu(t)=g(x)g(y)+\alpha^2 f(x)f(y), x,y\in S.$
This proves (\ref{63}).
\end{proof}
\begin{lem}\label{l2}
Let   $f, g :S\longmapsto\mathbb{C}$ be a solution  of the functional equation (\ref{elqorachi2}) such that  $\int_{S}  f(t)d\mu(t)=0$ and   $\{f,g \}$ is linearly independent.  \\
		 							 	 	(i) $\int_{S}  g(t) d\mu(t)\neq0.$  \\
		 							 	 	(ii)  If $\int_{S} \int_{S}  	   g(\sigma(t)s)d\mu(s) d\mu(t)=0 $,
		 							 	 	then $\int_{S} \int_{S}  	   f(\sigma(t)s)d\mu(s) d\mu(t)\neq0.$
		 							 	 \end{lem}
		 							 	 \begin{proof}
		 							 	 	(i) Assume that $\int_{S}  g(t) d\mu(t)=0$. Making    the substitution  $(x,\sigma(ys) )$  in (\ref{elqorachi2}) and integrating the result obtained with respect to $s$ we get
		 							 	 	\begin{equation*}\begin{split}&					 	 		\int_{S}\int_{S}f(x\sigma(ys)t)d\mu(t)d\mu(s)\\
&=f(x )\int_{S}g(ys)d\mu(s)+g(x)\int_{S}f(ys)d\mu(s)\\
&=\int_{S}\int_{S}	f(x\sigma(y)\sigma(s)t)d\mu(t)d\mu(s)\\	&=f(x\sigma(y))\int_{S}g(s)d\mu(s)+g(x\sigma(y))\int_{S}  f(s)d\mu(s)=0.\end{split}\end{equation*}
Since $\{f,g\}$ is linearly independent, then in particular     $\int_{S}f(ys)d\mu(s) =0 $ for all $y\in S$, and  Eq.(\ref{elqorachi2}) implies that
		 						$f(x)g(y)=-f(y)g(x)$ for all $ x,y\in S.
$ This contradicts the fact that
$\{f,g\}$ linearly independent. So, $\int_{S}  g(t) d\mu(t)\neq0$.\\
(ii) Assume that $\int_{S}\int_{S}g(\sigma(t)s)d\mu(s )d\mu(t)=0$.\\ If $\int_{S}\int_{S}f(\sigma(t)s)d\mu(s )d\mu(t )=0$ then equation (\ref{14}) can be written  as follows
$		 							 	 	f(x)g(y)+f(y)g(x)=0$ for all $x,y\in S.
$ 	This contradicts the fact that $\{f,g\}$ is linearly independent. Thus  $\int_{S}\int_{S}f(\sigma(s)t)d\mu(s)d\mu(t ) \neq0$.
\end{proof}
		 							 	 \begin{lem}
		 							 	 	\label{l3}
		 							 	 	Let   $f, g :S\longmapsto\mathbb{C}$ be a solution  of  equation (\ref{elqorachi2}) such that  $\int_{S}  f(t)d\mu(t)=0$ and   $\{f,g \}$ is linearly independent. \\
		 							 	 	(i)	$   \int_{S} f(ts)d\mu(s)d\mu(t)= \int_{S}\int_{S} f(\sigma(t)s)d\mu(s)d\mu(t)=\int_{S}\int_{S}f(s\sigma(t))d\mu(t)d\mu(s).$
		 							 	 	(ii) If $\int_{S} \int_{S}  	   g(\sigma(t)s)d\mu(s) d\mu(t)\neq0$, then   $\int_{S} \int_{S}  	   f(\sigma(t)s)d\mu(s) d\mu(t)=0.$
		 							 	 \end{lem}
		 							 	 \begin{proof} (i)
		 							 	     By using  (\ref{elqorachi2})   we have
\begin{equation*}\begin{split}& \int_{S}\int_{S}\int_{S} \int_{S}f(r\sigma(ts)k)d\mu(k)d\mu(s)d\mu(t)d\mu(r)\\
&=\int_{S}f(r)d\mu(r)\int_{S}\int_{S}g(ts)d\mu(s)d\mu(t)+\int_{S}g(r)d\mu(r)\int_{S}\int_{S}f(ts)d\mu(s)d\mu(t)\\  &=   \int_{S}g(r)d\mu(r)\int_{S}\int_{S}f(ts)d\mu(s)d\mu(t)\\ &=\int_{S}\int_{S}\int_{S} \int_{S}	f(r\sigma(t)\sigma(s)k)d\mu(k)d\mu(s)d\mu(t)d\mu(r)\\
&=\int_{S}f( s)d\mu(s) \int_{S}\int_{S}g(r\sigma(t))d\mu(t)d\mu(r)\\
&\quad+\int_{S}g(s)d\mu(s)\int_{S}\int_{S}f(r\sigma(t))d\mu(t)d\mu(r)\\
&= \int_{S}g( s)d\mu(s) \int_{S}\int_{S}f(r\sigma(t))d\mu(t)d\mu(r).\end{split}\end{equation*}
This implies that
		 							 	 	\begin{equation}
		 							 	 	\label{r1}
		 							 	 	\int_{S}\int_{S}f(s\sigma(t))d\mu(t)d\mu(s)= \int_{S}\int_{S}f(ts)d\mu(s)d\mu(t),
\end{equation}
since $  \int_{S}g( s)d\mu(s)\neq 0$ by Lemma \ref{l2} (i). On the other hand by using  (\ref{elqorachi2})  we have
\begin{equation*}\begin{split}&\int_{S}\int_{S}\int_{S} \int_{S}f(r\sigma(\sigma(t)s)k)d\mu(k)d\mu(s)d\mu(t)d\mu(r)\\
&=\int_{S}f(r)d\mu(r)\int_{S}\int_{S}g(\sigma(t)s)d\mu(s)d\mu(t)+\int_{S}g(r)d\mu(r)\int_{S}\int_{S}f(\sigma(t)s)d\mu(s)d\mu(t)\\
&=\int_{S}g(r)d\mu(r)\int_{S}\int_{S}f(\sigma(t)s)d\mu(s)d\mu(t)\\  &=\int_{S}\int_{S}\int_{S} \int_{S}f(rt\sigma(s)k)d\mu(k)d\mu(s)d\mu(t)d\mu(r)\\
&=\int_{S}f( s)d\mu(s) \int_{S}\int_{S}g(rt)d\mu(t)d\mu(r)+ \int_{S}g( s)d\mu(s) \int_{S}\int_{S}f(rt)d\mu(t)d\mu(r)\\
&=\int_{S}g(s)d\mu(s)\int_{S}\int_{S}f(rt)d\mu(t)d\mu(r),
\end{split}\end{equation*}
from which we derive		 							 	 	 \begin{equation}
\label{r2}		 							 	 	 \int_{S}\int_{S}f(ts)d\mu(s)d\mu(t)= \int_{S}\int_{S}f(\sigma(t)s)d\mu(s)d\mu(t),
\end{equation}
since $  \int_{S}g( s)d\mu(s)\neq 0$ by Lemma \ref{l2} (i).
Then (\ref{r1})  and (\ref{r2})  prove (i).
\\(ii)  Suppose $\int_{S}\int_{S} g(\sigma(t)s) d\mu(s) d\mu(t)\neq0$ and  $\int_{S}\int_{S} f(\sigma(t)s) d\mu(s) d\mu(t)\neq0$ . 	Letting $y=\sigma(t)$ in (\ref{14}) and integrating the result obtained with respect to $t$, we have
		 							 	 	\begin{equation*}
		 							 	 	\int_{S} \int_{S}  	   g(\sigma(t)s)d\mu(s) d\mu(t) \int_{S}f(xt)d\mu(t)=\end{equation*}   $$f(x) \int_{S}  g(\sigma(t))d\mu(t)   \int_{S}  	   g(s)d\mu(s) +g(x) \int_{S}  f(\sigma(t))d\mu(t)   \int_{S}  	   g(s)d\mu(s) $$$$   -\int_{S}\int_{S}    f(\sigma(t)s)d\mu(s)  d\mu(t) \int_{S}g(xt)d\mu(t).$$
		 							 	 	By setting $x=s$ in the last   identity  and integrating the result obtained  with respect to $s$  and using (i), we get that
		 							 	 	\begin{equation}\label{s1}\begin{split}&
\int_{S}f(\sigma(t))d\mu(t) \left( \int_{S}  	   g(s)d\mu(s)\right)^2\\		 							 	 	&=\int_{S}\int_{S}    f(\sigma(t)s)d\mu(s)  d\mu(t)\left[ \int_{S}\int_{S}    g(\sigma(t)s)d\mu(s)  d\mu(t)+\int_{S} \int_{S}g(st)d\mu(t)d\mu(s)\right].\end{split}\end{equation}
On the other hand, by  using (\ref{elqorachi2})  we get that
		 							 	 	$$\int_{S}	 \int_{S}	\int_{S} f(x\sigma(yks)t)d\mu(t)d\mu(s)d\mu(k) $$
		 							 	 	\begin{eqnarray*}		
		 							 	 		&=&     f(x )\int_{S} \int_{S}  	   g(yks)d\mu(s)    d\mu(k)+   g(x  )\int_{S}\int_{S}  f(yks)d\mu(s)d\mu(k)\\
		 							 	 		&=&  f(x )\int_{S}  \int_{S} 	   g(yks)d\mu(s)    d\mu(k)\\
		 							 	 		&\quad\quad+& g(x)  \big[ g(y  )\int_{S}f(\sigma(k)) d\mu(k)+f(y)\int_{S}g(\sigma(k)) d\mu(k)\big]
\end{eqnarray*}
and
\begin{equation*}\begin{split}&\int_{S} \int_{S}	\int_{S} f(x\sigma(yks)t)d\mu(t)d\mu(s)d\mu(k)\\
&=\int_{S} \int_{S}	\int_{S} f(x\sigma(y)\sigma(ks)t)d\mu(t)d\mu(s)d\mu(k)\\
&= f(x\sigma(y ) )\int_{S}\int_{S} g( zs) d\mu(s)d\mu(k) +  g(x\sigma(y ) ) \int_{S} \int_{S}f(ks) d\mu(s)d\mu(k).\end{split}\end{equation*}
We infer from the two last identities that
		 							 	 	$$f(x )\int_{S}\int_{S}  	   g(yks)d\mu(s)    d\mu(z)+g(x) \left[  g(y  )\int_{S}\int_{S}f(\sigma(k)) d\mu(k)+f(y)\int_{S}g(\sigma(k)) d\mu(k)\right]  $$ $$= f(x\sigma(y ) )\int_{S}\int_{S} g(ks) d\mu(s)d\mu(k) +  g(x\sigma(y ) ) \int_{S}\int_{S} f(ks) d\mu(s)d\mu(k).$$
Letting  $x=s$ and $y=t$ in the last equation, and 	integrating the result obtained with respect to $t$  and  $s$, by using (i) and $\int_S f(t)d\mu(t)=0$ we deduce that
		 							 	 	\begin{equation}
		 							 	 	\label{s2}
		 							 	 	\int_{S}  	   f(\sigma(t))d\mu(t)   \left( \int_{S}  	   g(s)d\mu(s)\right)^2
		 							 	 	\end{equation}     $$=\int_{S}\int_{S}    f(\sigma(t)s)d\mu(s)  d\mu(t)\left[ \int_{S}\int_{S}    g(ks)d\mu(s)  d\mu(k)+\int_{S} \int_{S}g(k\sigma(s))d\mu(s)d\mu(k)\right].$$
		 							 	 	Then, by combining (\ref{s1}) and (\ref{s2}) and using (i) we get
		 							 	 	\begin{equation}
		 							 	 	\label{s4}
		 							 	 	\int_{S}\int_{S}    g(\sigma(t)s)d\mu(s)  d\mu(t)=\int_{S} \int_{S}g(t\sigma(s))d\mu(s)d\mu(t),
\end{equation}
since $	\int_{S}\int_{S} f(\sigma(t)s)d\mu(s)d\mu(t)\neq0$. Now,  putting $x=t$ and $y=s$ in (\ref{14}), and  integrating the result with respect to $s$ and  $t$, by using (i), we get that
		 							 	 	\begin{equation}
		 							 	 	\label{s3}
		 							 	 	\int_{S} \int_{S}  	   g(\sigma(t)s)d\mu(s) d\mu(t)=-\int_{S} \int_{S}  	   g(t\sigma(s) )d\mu(s) d\mu(t),
		 							 	 	\end{equation}
		 							 	 	since $\int_{S}\int_{S} f(\sigma(t)s)d\mu(s)d\mu(t)\neq0$.
		 							 	 	Then, from (\ref{s4}) and (\ref{s3}) we derive  $$\int_{S} \int_{S}  	   g(\sigma(t)s)d\mu(s) d\mu(t)=\int_{S} \int_{S}  	   g(t\sigma(s) )d\mu(s) d\mu(t)=0.$$
This contradicts the assumption that $\int_{S}\int_{S} g(\sigma(t)s) d\mu(s) d\mu(t)\neq0$, and completes the proof.
		 							 	 \end{proof} Now, we are ready to prove the main result of this section.\begin{thm}
\label{t1} The solutions $f,g:S\longrightarrow \mathbb{C}$ of   functional equation (\ref{elqorachi2}) are the following pairs of   functions.\\
(1) $f=0$ and $g$ arbitrary. \\
(2)   $f\neq0$, $g=0$ and $\int_{S} f(xyt) d\mu(t)=0$ for all $x,y \in S.$\\
		 							 	 	(3) There exist a constant $\delta\in \mathbb{C}^*$ and an exponential   $\chi $   on $S$ such that
		 							 	 	$ \chi\circ\sigma=\chi$,	$\int_{S} \chi(t) d\mu(t)\neq0$ and $
		 							 f=\dfrac{\chi}{2\delta} \int_{S} \chi(t) d\mu(t)$  and   $g= \dfrac{\chi}{2}  \int_{S} \chi(t) d\mu(t).$\\
		 							 	 	(4) There exist a constant $\alpha \in \mathbb{C}^*$ and two different exponentials   $\chi_1$
		 							 	 	and $\chi_2$
		 							 	 	on $S$  with $\chi_1\circ\sigma=\chi_1, \chi_2\circ\sigma=\chi_2,$  $\int_{S} \chi_1(t) d\mu(t)  \neq0$ and
		 							 	 	$\int_{S} \chi_2(t) d\mu(t) \neq0$    such that
		 							 	 $
		 							 	 	f= \dfrac{1  }{2\alpha}\left[ \chi_1\int_{S} \chi_1(t) d\mu(t) -   \chi_2\int_{S} \chi_2(t) d\mu(t)\right]  $ and \\
		 							 	 	$g=\dfrac{1  }{2 }\left[ \chi_1\int_{S} \chi_1(t) d\mu(t) +  \chi_2\int_{S} \chi_2(t) d\mu(t)\right].
		 							 	 $\\
		 							 	 	(5) There exist an  exponential $\chi $ on $S$	 and a non-zero function  $\phi_{\chi}$ solution of (\ref{111}) with  $\chi\circ\sigma=\chi$,  $\int_{S}\chi(t)d\mu(t)\neq0$,
	 							 	 	$\int_{S}\phi_\chi(t)d\mu(t)=0$, and $\phi_{\chi}\circ\sigma=\phi_{\chi}$  such that      $f=\phi_{\chi}$  and $ g=\chi	\int_{S}\chi(t)d\mu(t).$  \\(6) There exist an  exponential   $\chi $  on $S$	 and a non-zero function  $\Phi_{\chi}$ solution of   (\ref{44}) with $\int_{S}\chi(t)d\mu(t)\neq0$,    $\chi\circ\sigma=\chi$ such that      $f=\Phi_{\chi}$  and $g=\chi	\int_{S}\chi(t)d\mu(t).$\\
		 							 	 	Moreover if $S$ is a topological semigroup, $f\neq0$ and $f \in  {C}(S)$ then $g, \chi,\chi_1, \chi_1$ and $ \Phi_{\chi}   \in  {C}(S)$.\\
		 							 	 	Conversely, the solutions listed above for $f$ and $g$ define  a solution of (\ref{elqorachi2}).
		 							 	 \end{thm}
		 							 	 \begin{proof}
		 							 	 	Clearly if $f = 0$ then $g$ is arbitrary, and this is case (1). Henceforth
		 							 	 	we assume $f\neq0$.  Next suppose there exists a constant $ \delta \in \mathbb{C}$ such that $g=\delta f$. Then   equation (\ref{elqorachi2}) reduces to
		 							 	 	$\int_{S}f(x\sigma(y)t)d\mu(t)=2\delta f(x)f(y)$ for all  $x,y \in S.$
		 							 	 	If $\delta=0$ then $g=0$ and $\int_{S}f(x\sigma(y)t)d\mu(t)=0$ for all $x, y \in S$, and  we see that we deal with case (2).
		 							 	 	If $\delta \neq0$  then in view of  Proposition \ref{p1}  there exists  an exponential   $\chi$ on $S$ such that $2\delta f=: \chi \int_{S}\chi(t)d\mu(t)$ with $\chi\circ\sigma=\chi$, $\int_{S}\chi(t)d\mu(t)\neq0$. This implies
		 							 	 	$f=\dfrac{\chi }{2\delta } \int_{S}\chi(t)d\mu(t)$ and $g=\dfrac{\chi}{2}\int_{S}\chi(t)d\mu(t),$
		 							 	 	which is a solution family (3). From now we assume
		 							 	 	that $f$ and $g$ are linearly independent, and we consider two cases.\\
		 							 	 	 Case A:  $\int_{S}f(t)d\mu(t)=0 $.  From  Lemma \ref{l2} (i) we have: $\int_{S}g(t)d\mu(t) \neq0$.\\
		 							 	 	 \underline{Subcase A.1}:  $\int_{S}\int_{S}g(\sigma(t)s)d\mu(s)d\mu(t) =0$. From Lemma \ref{l2}  (ii)  we have      $\int_{S}\int_{S}f(\sigma(t)s)d\mu(s)d\mu(t)\neq0$. Then
		 							 	 	equation (\ref{14}) now reads
		 							 	 	\begin{equation}
		 							 	 	\label{e1}
		 							 	 	g(x\sigma(y))=  \beta g(x)f(y) +\beta g(y)f(x) , \; x,y\in S,
		 							 	 	\end{equation}
		 							 	 	where $\beta:=\dfrac{\int_{S} g(s)d\mu(s) }{\int_{S}\int_{S}f(\sigma(t)s)d\mu(s)d\mu(t)}\neq0$. By applying \cite[Theorem 4.3 ]{as} to (\ref{e1}) and using that $\{f,g\}$ is   independent we have only the following two possibilities:\\
		 							 	    (i) There exist  a constant $c\in \mathbb{C}^{*}$ and two   exponentials  $\chi_1$ and $\chi_2$ on $S$ such that: $\chi_1\neq \chi_2$, $\chi_1\circ \sigma=\chi_1$, $\chi_2\circ \sigma=\chi_2$ and
		 							 	 	$\beta f = (\chi_1+\chi_2)/{2},$ $g=c(\chi_1 - \chi_2)$. Substituting these new forms of $f$ and $g$ into (\ref{elqorachi2})  we get after some calculations that
		 							 	 	\begin{equation*}
		 							 	 	\left(  2 c -  \int_{S}\chi_1(t)d\mu(t)\right) \chi_1(xy)- \left( 2 c +  \int_{S}\chi_2(t)d\mu(t) \right) \chi_2(xy)=0,\;\; x,y \in S.
\end{equation*}
By \cite[Theorem 3.18]{g}  we get that   $\int_{S}\chi_1(t)d\mu(t)=-\int_{S}\chi_2(t)d\mu(t)=2c$. So we have  the special case of  the  solution family (4) corresponding to $\int_{S}\chi_1(t)d\mu(t)=-\int_{S}\chi_2(t)d\mu(t)=\beta^{-1}\alpha$.
		 							 	 	\\(ii)  There exist an  exponential  $\chi $   on $S$ such that:   $\beta f=\chi $  and  $g=\phi_{\chi}=\phi_{\chi}\circ \sigma$. Moreover $\phi_{\chi}\neq0$ since $\{f, g\}$ is independent. A small computation based on (\ref{elqorachi2}) show that
		 							 	 	\begin{equation}\label{ebn}
		 							 	 	\chi(x)\left(\phi_{\chi}(y)-\chi(y)\int_{S}\chi(t)d\mu(t)  \right)+\phi_{\chi}(x)\chi(y) =0,\;\; x,y \in S.
		 							 	 	\end{equation}
		 							 	  Then, by using \cite[Lemma 5.1]{d}, we get that $\chi=\phi_{\chi}=0$.
		 							 	 	This case does not apply, because $\chi\neq0$ and $\phi_{\chi}\neq0$.\\
		 							 	 	\underline{Subcase A.2}: Suppose $\int_{S}\int_{S}g(\sigma(t)s)d\mu(s)d\mu(t)\neq0$, then by lemma \ref{l3}  (ii)  we read that     $\int_{S}\int_{S}f(\sigma(t)s)d\mu(s)d\mu(t)=0$, and then equation (\ref{14}) yields
		 							 	 	\begin{equation}
		 							 	 	\label{e2}
		 							 	 	f(x\sigma(y))=  \alpha f(x)g(y) +\alpha f(y)g(x) , \; x,y\in S,
		 							 	 	\end{equation}	
		 							 	 	where $\alpha:=\dfrac{\int_{S} g(s)d\mu(s) }{\int_{S}\int_{S}g(\sigma(t)s)d\mu(s)d\mu(t)}\neq0$. By applying   \cite[Lemma 4.1]{as} to (\ref{e2}) we get that $f\circ\sigma = f$ and $g\circ\sigma = g$ since $\alpha\neq0$ and $\{f,g\}$ is independent. Therefore equation (\ref{e1}) is a sine addition law:	
		 							 	 	$f(xy)=  \alpha f(x)g(y) +\alpha f(y)g(x) , \; x,y\in S.
		 							 	 	$
		 							 	 	So we know from \cite[Theorem 4.3 ]{as} that there are only the following 2 possibilities:\\
		 							 	 	(i) There exists a constant $c\in \mathbb{C}^{*}$, there exist   two   exponentials  $\chi_1$ and $\chi_2$ on $S$ such that: $\chi_1\neq \chi_2$, $\chi_1\circ \sigma=\chi_1$, $\chi_2\circ \sigma=\chi_2$ and such that
		 							 	 	$\alpha g=( \chi_1+\chi_2){/2}$ and $f=c(\chi_1-\chi_2)$. Substituting the new form of $f$ and $g$ into (\ref{elqorachi2}) we get after some rearrangement  that
		 							 	 	\begin{equation*}
		 							 	 	\left(\alpha \int_{S}\chi_1(t)d\mu(t)-1 \right)\chi_1(xy)+\left(1-\alpha \int_{S}\chi_2(t)d\mu(t)+ \right)\chi_2(xy)=0
		 							 	 	\end{equation*} for all $ x,y\in S.$
		 							 	 	In view of \cite[Theorem 3.18 ]{g}  we deduce that
		 							 	  $\int_{S}\chi_1(t)d\mu(t)=\int_{S}\chi_2(t)d\mu(t)= {1}/{\alpha} $. Then $g=  \big[\int_{S}\chi_1(t)d\mu(t)\big](\chi_1+\chi_2)  / 2$. Therefore we get the special case  of the solution (4) corresponding to   $\int_{S}\chi_1(t)d\mu(t)=\int_{S}\chi_2(t)d\mu(t)= 2\alpha c$.\\
		 							 	 	(ii) There exist an  exponential  $\chi $   on $S$ such   $\alpha g=\chi$  and $f=\phi_{\chi}=\phi_{\chi}\circ \sigma$. Furthermore $\phi_{\chi}\neq0$ since $\{f, g\}$  is independent.
		 							 	 	Then by substituting the expression of $f$ and $g$ into (\ref{elqorachi2}), we get
		 							 	 	$$\phi_{\chi}(x)\left( \chi(y)\left[\int_{S}\chi(t)d\mu(t)-\dfrac{1}{\alpha} \right] \right)$$ $$+\chi(x)\left(\phi_{\chi}(y)\left[\int_{S}\chi(t)d\mu(t)-\dfrac{1}{\alpha} \right] +\chi(y) \int_{S}\phi_{\chi}(t)d\mu(t)\right)=0, \;\; x,y \in S,$$which implies,  by using  \cite[Lemma 5.1   (b)]{d}, that
		 							 	 	$ \int_{S}\chi(t)d\mu(t)=\dfrac{1}{\alpha}$ and $\int_{S}\phi_{\chi}(t)d\mu(t)=0.$
		 							 	 	This gives the solution of   case (5).
		 							 	 	\\{Case B}: Suppose that   $\int_{S}f(t)d\mu(t)\neq0$. The systems of equations
		 							 	 	combining each pair of (\ref{elqorachi2}), (\ref{63}) implies that   $
		 							 	 	\int_{S}(g+\alpha f)(x\sigma(y)t)d\mu(t)=(g+\alpha f)(x)	(g+\alpha f)(y)
		 							 	 	$, and $
		 							 	 	\int_{S}(g-\alpha f)(x\sigma(y)t)d\mu(t)=(g-\alpha f)(x)	(g-\alpha f)(y)$ for all $x,y\in S$.
		 							 	 	According to  Proposition \ref{p1}   there exist  two multiplicative functions  $\chi_1$ and $\chi_2$  on $S$ such that $\chi_1\int_{S}\chi_1(t)d\mu(t):= g+\alpha f$ and $\chi_2 \int_{S}\chi_1(t)d\mu(t):=	g-\alpha f,$
		 							 	 	with $\chi_1\circ \sigma=\chi_1$, $\chi_2\circ \sigma=\chi_2$. Furthermore,  $\chi_1$ and $\chi_2$ are exponentials, $ \int_{S}\chi_1(t)d\mu(t) \neq0 $ and $ \int_{S}\chi_2(t)d\mu(t)  \neq0 $,   since $\{f,g\}$ is   linearly independent.\\If $ \chi_1\int_{S}\chi_1(t)d\mu(t)\neq\chi_2\int_{S}\chi_2(t)d\mu(t)$, then    $\alpha\neq0$, and  	we deal with case (4).\\If $ \chi_1\int_{S}\chi_1(t)d\mu(t)= \chi_2 \int_{S}\chi_2(t)d\mu(t)$,  then   from \cite[Theorem 3.18]{g}   we have $ \chi_1=\chi_2$   since  $ \int_{S}\chi_1(t)d\mu(t)\neq0$ and $\int_{S}\chi_2(t)d\mu(t)\neq0$. Let $\chi:= \chi_1= \chi_2$ with $\chi$    exponential such that $\chi\circ\sigma=\chi$ and  $\int_{S}\chi (t)d\mu(t)\neq0$, we get $  g=  \chi\int_{S}\chi (t)d\mu(t)$, and $f=\Phi_\chi$ satisfies (\ref{44}).
		 							 	   This solution     occurs in  part (6).
		 							 	 	Conversely it is easy to check that the formulas of $f$ and $g$ listed in  Theorem 3.4 define  solutions of (\ref{elqorachi2}).	Finally,  suppose that  $S$ is a   topological semigroup. The cases (1)-(5) can be treated as  cases (1)-(5) of \cite[Theorem 4.4]{f}. In case (6)  we have  $0\neq f=\Phi_\chi \in C(S)$ is a solution of (\ref{44}) and since $\int_{S} \chi(t)d\mu(t)\neq0$ in this case we get
		 							 	 	$$\chi(x)=\dfrac{\int_{S} f(x\sigma(q)t)d\mu(t)-  f(x)\chi(q) \int_{S} \chi(t)d\mu(t)}{  f(q) \int_{S} \chi(t)d\mu(t)},  $$
		 							 	 	for some $q\in S$ such that  $f(q)\neq0$. Since $f, \sigma\in C(S)$  and  $\mu$ is a discrete measure we get that $\chi \in C(S)$.
		 							 	 	
		 							 	 	\end{proof}


\begin{thebibliography} {99}
  		\bibitem{a} {  Acz\'{e}l, J.,  Dhombres, J.} { Functional Equations in sevaral variables. With Applications to mathematics, information theory and to the natural and social sciences}, Encyclopedia of Mathematics and its Applications, 31 Cambridge University Press, Cambridge. 1989.
  			\bibitem{ajjj}   {Ajebbar, O.,   Elqorachi, E.}	{The Cosine-Sine functional equation on a semigroup with an involutive automorphism}.  Aequat. Math. 91, 1115-1146 (2017). \url{https://doi.org/10.1007/s00010-017-0512-9}
  		\bibitem {as}{ Aserrar, Y.,     Elqorachi, E.} {Cosine and Sine addition and subtraction law with an automorphism.}
  		Annales Math. Silesianae (2023).DOI: \url{https://doi.org/10.2478/amsil-2023-0021.}
  		\bibitem{ajj}  {Ajebbar,  O.,  Elqorachi,  E.} 	{ Solutions and stability  of trigonometric functional equations on an amenable group with an involutive automorphism.}  Communications  Korean Math. Soc., 34 (1), 55-82, (2019).
  		\bibitem {akk1} {Akkouchi, M.,   Bouikhalene, B., Elqorachi, E.}{  Functional Equations and $\mu$-Spherical Functions.} Georgian Mathematical Journal, vol. 15, no. 1, 2008, pp. 1-20. \url{https://doi.org/10.1515/GMJ.2008.}
  		\bibitem {akk2} {Elqorachi, E.,   Akkouchi, M.,  Bakali, A.,   Bouikhalene, B.} {Badora's Equation on Non-Abelian Locally Compact}. Groups Georgian Math. J. Vol. 11, no. 3, 2004, pp. 449-466. \url{https://doi.org/10.1515/GMJ.2004.}
  		\bibitem {akk3} {Elqorachi, E.,   Akkouchi, M.}  {  On Hyers-Ulam Stability of Cauchy and Wilson.}
  		Equations Georgian Math. J.
  		Volume 11 (2004) , Number 1, 69-82.
  	\bibitem {ass} { Aserrar, Y.,   Chahbi, A.,  Elqorachi, E.} {   A variant of Wilson's functional equation on semigroups.} Commun. Korean Math. Soc. 2023, 38(4),
  		 1063-1074.
  		 \bibitem{q} {Aserrar, Y.,  Elqorachi, E.}  { A d'Alembert type functional equation on semigroups}. \url{http://arxiv.org/abs/2210.09111v1.}
  		  \bibitem{boui} { Bouikhalene, B.,  Elqorachi, E.} {  An extension of Van Vleck’s functional equation for the sine.}  Acta Math. Hung.,  150,  258-267, (2016).
  			\bibitem{eb} Ebanks, B.,   Stetk\ae r, H.  { D'Alembert's other
  				functional equation on monoids with an involution}.  Aequat. Math. 89, 187-206 (2015). \url{https://doi.org/10.1007/s00010-014-0303-5}
  		  	\bibitem{c} { Ebanks, B.} {  The sine addition and subtraction formulas on semigroups.}  Acta Math. Hungar.,  167 (2), 533-555, (2021).
  		  	 \bibitem{d} { Ebanks, B.} { Around the Sine Addition Law and d'Alembert's Equation on Semigroups}. Results Math 77, 11 (2022). \url{https://doi.org/10.1007/s00025-021-01548-6}
    \bibitem{vb} {Ebanks, B.} { Sine Subtraction Laws on Semigroups}. Annales Math. Silesianae 37(2023), no. 1, 49–66.
     \bibitem{red}   Elqorachi, E.,  Redouani, A. Solutions and stability of generalized Kannappan's and Van Vleck's equations. Annales Math. Silesianae 32 (2018), 169-200.
     \bibitem{Redouani}  Elqorachi, E.,  Redouani, A. Trigonometric formulas and $\mu$-spherical functions, Aequationes Math. V. 72, pages 60-77, (2006).
    \bibitem{f} Jafar, A., Ajebbar, O. and  Elqorachi, E., A Kannappan-sine addition law on semigroups. Aequat. Math. 98, 1001-1017 (2024). \url{https://doi.org/10.1007/s00010-024-01104-x}
    \bibitem{ibti} Jafar, A., Ajebbar, O.,  Elqorachi, E. A Kannappan-sine subtraction law on semigroups. Aequat. Math. (2024). https://doi.org/10.1007/s00010-024-01098-6
     \bibitem{ff}   Jafar, A., Ajebbar, O. and Elqorachi, E. A generalization of the Kannappan-sine addition law on semigroups. Aequat. Math. (2024). \url{https://doi.org/10.1007/s00010-024-01138-1}
    \bibitem{PL} {Kannappan, Pl.} { A functional equation for the cosine}, Canad. Math. Bull., 2, 495-498, (1968).
     \bibitem{pe}  Perkins, A.M.,   Sahoo, P.K. { On two functional equations with involution on groups related to sine and cosine functions}. Aequationes Math. 89 (2015), no. 5, 1251–1263
    \bibitem{ebb}  Poulsen, T.A.,   Stetk\ae r, H.  {On the trigonometric subtraction and addition formulas}. Aequ. math. 59, 84-92 (2000). \url{https://doi.org/10.1007/PL00000130}
   			\bibitem{g}  {  Stetk\ae r, H. } { Functional Equations on Groups. World Scientific Publishing Company}.
   			Singapore (2013).
   			\bibitem{stetk} { Stetk\ae r, H.} { A Levi-Civita functional equation on semigroups}, Aequationes. Math. 96, 115-127 (2022).
   			\bibitem{i} {  Stetk\ae r, H.} { Kannappan's functional equation on semigroups with involution.}
   			Semigroup Forum, 94 (1), 17-30, (2017).
   			\bibitem {l} {   Zeglami, D.,   Tial, M.,   Kabbaj, S.}, { The integral cosine addition and sine subtraction laws,}  Results 	Math. 73, no. 3, 9, (2018). 	
  \end{thebibliography}
\end{document}